\documentclass[12pt]{article}
\usepackage{latexsym, amssymb, amsthm, amsmath}
\usepackage{dsfont}

\DeclareMathAlphabet\gothic{U}{euf}{m}{n}

\makeatletter
\def\eqnarray{\stepcounter{equation}\let\@currentlabel=\theequation
\global\@eqnswtrue
\tabskip\@centering\let\\=\@eqncr
$$\halign to \displaywidth\bgroup\hfil\global\@eqcnt\z@
  $\displaystyle\tabskip\z@{##}$&\global\@eqcnt\@ne
  \hfil$\displaystyle{{}##{}}$\hfil
  &\global\@eqcnt\tw@ $\displaystyle{##}$\hfil
  \tabskip\@centering&\llap{##}\tabskip\z@\cr}

\def\endeqnarray{\@@eqncr\egroup
      \global\advance\c@equation\m@ne$$\global\@ignoretrue}

\def\@yeqncr{\@ifnextchar [{\@xeqncr}{\@xeqncr[5pt]}}
\makeatother

\allowdisplaybreaks

\newtheorem{lemma}{Lemma}[section]
\newtheorem{thm}[lemma]{Theorem}
\newtheorem{cor}[lemma]{Corollary}
\newtheorem{prop}[lemma]{Proposition}

\theoremstyle{definition}

\newtheorem{exam}[lemma]{Example}

\newcommand{\gota}{\gothic{a}}

\newcounter{teller}

\newenvironment{tabel}{\begin{list}%
{\rm  (\alph{teller})\hfill}{\usecounter{teller} \leftmargin=1.1cm
\labelwidth=1.1cm \labelsep=0cm \parsep=0cm}
                      }{\end{list}}

\newcounter{tellerr}

\newcounter{tellerrr}

\newcounter{proofstep}

\newcommand{\Ni}{\mathds{N}}

\newcommand{\Ri}{\mathds{R}}
\newcommand{\R}{\mathds{R}}
\newcommand{\Ci}{\mathds{C}}

\newcommand{\Li}{\mathds{L}}

\newcommand{\dom}{\mathop{\rm dom}}

\newcommand{\RRe}{\mathop{\rm Re}}
\newcommand{\IIm}{\mathop{\rm Im}}

\newcommand{\Tr}{{\mathop{\rm Tr \,}}}

\newcommand{\supp}{\mathop{\rm supp}}
\newcommand{\sgn}{\mathop{\rm sgn}}

\newcommand{\esssup}{\mathop{\rm ess\,sup}}

\newcommand{\loc}{{\rm loc}}

\newcommand{\one}{\mathds{1}}

\newcommand{\ch}{{\cal H}}

\newcommand{\cw}{{\cal W}}

\newcommand{\altnorm}[1]{{\left\vert\kern-0.25ex\left\vert\kern-0.25ex\left\vert #1 
    \right\vert\kern-0.25ex\right\vert\kern-0.25ex\right\vert}}

\begin{document}

\thispagestyle{empty}

\vspace*{1cm}
\begin{center}
{\large\bf On the numerical range of sectorial forms} \\[5mm]
\large  A.F.M. ter Elst, A. Linke and J. Rehberg

\end{center}

\vspace{5mm}

\begin{list}{}{\leftmargin=1.8cm \rightmargin=1.8cm \listparindent=10mm 
   \parsep=0pt}
\item
\small
{\sc Abstract}.
We provide a sharp and optimal generic bound for the angle of the sectorial 
form associated to a non-symmetric second-order elliptic differential operator
with various boundary conditions.
Consequently this gives an, in general, 
sharper $\ch^\infty$-angle for the $\ch^\infty$-calculus
on $L_p$ for all $p \in (1,\infty)$ if the coefficients are real valued.
\end{list}

\let\thefootnote\relax\footnotetext{
\begin{tabular}{@{}l}
\end{tabular}}

\section{Introduction} \label{Ssector1}

In the $L_2$-theory of second-order divergence form operators it is classical that 
the numerical range of the sesquilinear form 
$\gota \colon W^{1,2}(\Omega) \times W^{1,2}(\Omega) \to \Ci$ given by
\[
\gota(u,v) = \int_\Omega \mu \nabla u \cdot \overline {\nabla v }
\]
is contained in the sector with \mbox{(half-)angle} $\arctan \frac {M}{m}$, 
if the coefficient function $\mu$ admits
the uniform bound $M$ and ellipticity constant~$m$.  
Moreover, it is is well-known that the angle of the numerical range sector 
has implications for resolvent estimates 
and for the holomorphic calculus, both for the $L_2$-realisation of the 
elliptic operator and the $L_p$-realisation, see below.
Hence the question arises whether the above angle is optimal.
In this paper we show that one can improve the angle.
The expression we find is completely explicit 
in $M$ and $m$, uniform in all matrices with uniform bound $M$ and ellipticity 
constant~$m$.
Moreover,  it is optimal, see Example~\ref{xsector204}.
All of this allows in Corollary~\ref{csector203} to give a sharper estimate for the angle 
of the sector containing the numerical range.

Further, we provide resolvent decay for the operator $A_p$
which is associated with the form~$\gota$ on $L_p(\Omega)$, where $p \in (1,\infty)$.
Uniform resolvent estimates for the elliptic operators 
are important for the treatment of  nonautonomous
parabolic equations, see for example 
\cite{Tan2}, \cite{Kat12}, \cite{SobolevskiiR}, \cite{AcquistapaceTerreni}, \cite{Yos}.
For an alternative approach, not using the evolution system, see \cite{DaPratoSinestrari}.
This use of uniform resolvent estimates is standard nowadays, see 
\cite[Chapter~II]{Ama2}, \cite[Section~6.1]{Lun}.

In Section~\ref{Ssector3} we prove that the operator $A_p$ admits a 
bounded $\ch^\infty$-calculus with \mbox{(half-)angle} smaller than $\pi/2$. 
Using the better numerical range on $L_2(\Omega)$, we obtain a better $\ch^\infty$-angle 
on $L_p(\Omega)$ by applying the Crouzeix--Delyon theorem and a theorem of Kalton--Kunstmann--Weis.
This enables sharper estimates for the purely imaginary powers of the operators. 
Applying the Dore--Venni theorem one obtains, as a byproduct, even maximal 
parabolic regularity on $L_p(\Omega)$ for all $p \in (1,\infty)$.

In Section~\ref{Ssector4} we consider as in \cite{HKrR} or \cite{EMR}
an elliptic operator subject to 
mixed boundary conditions and domain inhomogeneities supported on a lower
dimensional hypersurface.
Such results are of use when treating parabolic problems with dynamical boundary conditions, 
compare also \cite{VV}.

\section{\hspace*{-6pt}Numerical range and spectral theoretic consequences} \label{Ssector2}

Let $H$ be a Hilbert space with $H \neq \{ 0 \} $ and let $T$ be 
an operator in $H$ with domain $\dom(T)$. 
The {\bf numerical range} $\Lambda(T)$ of $T$ is defined by
\[
\Lambda(T)
= \{ (Tu, u)_H : u \in \dom(T) \mbox{ and } \|u\|_H = 1 \} 
 .  \]
A classical theorem of Hausdorff says that the numerical range is a convex set.
For all $\theta \in [0,\pi)$ define
\[
\Sigma(\theta)
= \{ r \, e^{i \varphi} : r \in [0,\infty) \mbox{ and } \varphi \in [-\theta,\theta] \} 
 .  \]
Then $\Sigma(\theta)$ is closed and $0 \in \Sigma(\theta)$.

\begin{prop} \label{psector201}
Let $T$ be a bounded operator in a Hilbert space $H$.
Let $E = \frac{1}{2i} (T - T^*)$ be the imaginary part of $T$.
Suppose that $T$ is coercive and let $m > 0$ be such that 
\[
\RRe (T u, u)_H
\geq m \, \|u\|_H^2
 .  \]
for all $u \in H$.
Then 
\begin{equation}
\|E\| \leq \sqrt{\|T\|^2 - m^2}
\label{epsector201;1}
\end{equation}
and 
$\Lambda(T) \subset \Sigma(\arctan \sqrt{\bigl (\frac {\|T\|}{m}\bigr )^2 -1})$.
\end{prop}
\begin{proof}
Define $S = \frac{1}{2} (T + T^*)$, the real part of $T$.
Then $S$ is self-adjoint and $(S u, u)_H = \RRe (T u, u)_H \geq m \, \|u\|_H^2$
for all $u \in H$.
Hence $\|S u\| \geq m \, \|u\|_H$ for all $u \in H$.

Note that the operator $E$ is self-adjoint.
First suppose that the operator $E$ has an eigenvalue such that the modulus 
is equal to $\|E\|$, that is, there exist $u \in H$ and $\lambda \in \Ri$
such that $E u = \lambda u$, $\|u\|_H = 1$ and $|\lambda| = \|E\|$.
Then
\begin{eqnarray*}
\|T\|^2
& \geq & \|T u\|^2
= ((S + i E)u, (S + i E)u)_H  \\
& = & \|S u\|^2 + \|E u\|^2 - i (Su, Eu)_H + i (Eu, Su)_H  \\
& = & \|S u\|^2 + \|E\|^2 - i \, \lambda \, (Su, u)_H + i \, \lambda \, (u, Su)_H  \\
& = & \|S u\|^2 + \|E\|^2 
\geq m^2 + \|E\|^2
,
\end{eqnarray*}
which implies (\ref{epsector201;1}).

Now we consider the general case.
Let $\varepsilon > 0$.
It follows from the spectral theorem that there exists a 
self-adjoint bounded operator $P$ such that 
$\|P\| \leq \varepsilon$ and the operator $E + P$
has an eigenvalue such that the modulus equals $\|E + P\|$.
Apply the above to the operator $S + i (E + P)$ and note that 
this operator has the same coercivity constant $m$.
One obtains the estimate
\[
\|E\| - \varepsilon
\leq \|E + P\| 
\leq \sqrt{\|T + i P\|^2 - m^2}
\leq \sqrt{(\|T\| + \varepsilon)^2 - m^2}
 .  \]
Finally take the limit $\varepsilon \downarrow 0$.

The inclusion is easy since
\[
|\IIm (Tu, u)_H|
= |(Eu, u)_H|
\leq \|E\| \, \|u\|^2
\leq \frac{ \sqrt{\|T\|^2 - m^2} }{ m } \, \RRe (Tu,u)
\]
for all $u \in H$.
\end{proof}

The estimate (\ref{epsector201;1}) is sharp.
Equality occurs  for example if $T = I + i E$, where $E$ is a bounded self-adjoint operator.

We apply Proposition~\ref{psector201} to sectorial 
forms associated to second-order differential operators.

\begin{thm} \label{tsector202}
Let $\Omega \subset \Ri^d$ be open and $\mu \colon \Omega \to \Ci^{d \times d}$
be a bounded measurable function.
Suppose that there exists an $m > 0$ such that 
$\RRe \mu(x) \xi \cdot \overline \xi \geq m \, |\xi|^2$
for all $\xi \in \Ci^d$.

Define the sesquilinear form 
$\gota \colon W^{1,2}(\Omega) \times W^{1,2}(\Omega) \to \Ci$ by 
\[
\gota[u,v]= \int_\Omega \mu \nabla u \cdot \overline{\nabla v}
 .  \] 
Then 
\begin{equation} \label{etsector202;5}
\gota[u] \in \Sigma(\kappa) 
\end{equation}
for all $u \in \dom(\gota)$,
where $\kappa = \arctan \sqrt{\bigl (\frac {M}{m}\bigr)^2- 1}$ and
$M= \esssup_{x \in \Omega} \|\mu(x)\|$.
Here $\|\cdot\|$ is the usual operator norm on~$\mathcal L(\Ci^d)$. 
\end{thm}
\begin{proof}
Let $u \in W^{1,2}(\Omega)$.
If $x \in \Omega$, then one can apply Proposition~\ref{psector201} to the 
operator $\mu(x)$ on~$\Ci^d$ to deduce that 
$(\mu \nabla u \cdot \overline{\nabla u})(x) 
\in \Sigma(\kappa)$.
Now integrate over $x \in \Omega$.
\end{proof}

The angle of the sector for a sectorial form gives resolvent bounds
for the associated operator, see \cite{Kat1} Theorem~V.3.2.
The following situation for $j$-elliptic forms gives a particularly
nice description. 
For $j$-elliptic forms and the associated m-sectorial operators we 
refer to \cite{AE2} Section~2.

\begin{thm} \label{tsector210}
Let $V$, $H$ be Hilbert spaces, $\gota \colon V \times V \to \Ci$ 
a sesquilinear form, $j \colon V \to H$ a continuous linear operator and 
$\kappa \in [0,\frac{\pi}{2})$.
Suppose that $\gota[u] \in \Sigma(\kappa)$ for all $u \in \dom(\gota)$
and that $j$ has dense range.
Further suppose that $\gota$ is $j$-elliptic.
Let $A$ be the operator associated with $(\gota,j)$.
Then $\sigma(A) \subset \overline{\Lambda(A)} \subset \Sigma(\kappa)$ and 
\[
\|(A + \lambda \, I)^{-1}\|_{H \to H}
\leq \frac{1}{d(-\lambda , \Sigma(\kappa))}
\]
for all $\lambda \in \Ci$ with $- \lambda \not\in \Sigma(\kappa)$.
\end{thm}
\begin{proof}
Since $\gota$ is $j$-elliptic, the operator $A$ is m-sectorial.
Hence $\sigma(A) \subset \overline{\Lambda(A)}$ by \cite{Kat1} Theorem~V.3.2.
Let $f \in \dom(A)$ with $\|f\|_H = 1$.
Then there exists a $u \in V$ such that $j(u) = f$ and 
$\gota[u,v] = (A f, j(v))_H$ for all $v \in V$.
Therefore $(Af, f)_H = \gota[u,u] \in \Sigma(\kappa)$.
So $\Lambda(A) \subset \Sigma(\kappa)$.
It follows from \cite{Kat1} Theorem~V.3.2 that 
$\|(A + \lambda \, I)^{-1}\|_{H \to H}
\leq \frac{1}{d(-\lambda , \overline{ \Lambda(A) } )}$
for all $\lambda \in \Ci$ with $-\lambda \not\in \overline{ \Lambda(A) } $.
This implies the inequality in the theorem.
\end{proof}

We return to second-order differential operators.

\begin{cor} \label{csector203}
Adopt the assumptions and notation as in Theorem~\ref{tsector202}. 
Let $V \subset W^{1,2}(\Omega)$ be a closed subspace such that $C_c^\infty(\Omega) \subset V$ 
and let $\gota_V = \gota|_{V \times V}$. 
Let $A_V$ be the m-sectorial operator associated with the form~$\gota_V$.
Then $\sigma(A_V) \subset \overline{\Lambda(A_V)} \subset \Sigma(\kappa)$.
Moreover, let $\theta \in (\kappa,\frac{\pi}{2})$.
Then 
\[
\|(A_V + \lambda \, I)^{-1}\|_{2 \to 2}
\leq \frac{M}{ m \, \sin \theta - \sqrt{M^2 - m^2} \, \cos \theta} \, \frac{1}{|\lambda|}
\]
for all $\lambda \in \Sigma(\pi - \theta)$.
\end{cor}
\begin{proof}
Evidently, $\gota_V[u] \in \Sigma(\kappa)$ for all $u \in V$
by Theorem~\ref{tsector202}. 
Now apply Theorem~\ref{tsector210} with $j \colon V \to L_2(\Omega)$ the 
identity map.
Then $\sigma(A_V) \subset \overline{\Lambda(A_V)} \subset \Sigma(\kappa)$ and 
\[
\|(A_V + \lambda \, I)^{-1}\|_{2 \to 2}
\leq \frac{1}{d(-\lambda , \Sigma(\kappa))}
\]
for all $\lambda \in \Ci$ with $- \lambda \not\in \Sigma(\kappa)$.
Then the assertion follows by elementary trigonometry.
\end{proof}

\begin{exam} \label{xsector204}
We present an example of an elliptic differential operator with real coefficients 
such that the angle $\kappa$ in (\ref{etsector202;5}) is optimal.

Let $\Omega \subset \Ri^2$ be a non-empty open bounded set.
Choose $\mu(x) = \left( \begin{array}{cc} 1 & 1 \\ -1 & 1 \end{array} \right)$
for all $x \in \Omega$ and $V = W^{1,2}(\Omega)$.
Define $u \in W^{1,2}(\Omega)$ by 
$u(x,y) = -x + y + i(x+y)$.
A straightforward calculation gives $\gota[u] = (4 - 4i) \, |\Omega|$.
Also $m = 1$ and $M = \sqrt{2}$.
So $\kappa = \frac{\pi}{4}$ and it cannot be improved.
\end{exam}

\section{Bounded $\ch^\infty$-calculus and maximal parabolic regularity} \label{Ssector3}

Let $\Omega \subset \Ri^d$ be open connected 
and $\mu \colon \Omega \to \Ri^{d \times d}$
be a bounded measurable function.
Suppose that there exists an $m > 0$ such that 
$\RRe \mu(x) \xi \cdot \overline \xi \geq m \, |\xi|^2$
for all $\xi \in \Ci^d$.
Let $M = \esssup_{x \in \Omega} \|\mu(x)\|$ and set
\[
\kappa = \arctan \sqrt{\Big(\frac {M}{m} \Big)^2 - 1}
 .  \]
In this section we consider realisations of the elliptic operator with 
(real) coefficients $\mu$
with mixed boundary conditions on the space $L_p(\Omega)$.

Define the sesquilinear form $\gota \colon W^{1,2}(\Omega) \times W^{1,2}(\Omega) \to \Ci$ by 
\[
\gota[u,v]= \int_\Omega \mu \nabla u \cdot \overline{\nabla v}
 .  \]
Let $D$ be a closed subset of $\partial \Omega$.
We denote by $W^{1,2}_D(\Omega)$ the closure in $W^{1,2}(\Omega)$ of the set
\[
\{ u|_\Omega : u \in C_c^\infty(\R^d) \mbox{ and } D \cap \supp u = \emptyset \} 
 .   \]
Moreover, define $\cw^{1,2}_D(\Omega)$ to be the closure of 
\[
\{u : u \in W^{1,2}(\Omega) \mbox{ and } D \cap \supp u = \emptyset \} 
\]
in $W^{1,2}(\Omega)$.
Define $\gota_W = \gota|_{W^{1,2}_D(\Omega) \times W^{1,2}_D(\Omega)}$ and 
$\gota_\cw = \gota|_{\cw^{1,2}_D(\Omega) \times \cw^{1,2}_D(\Omega)}$.
Then $\gota_W$ and $\gota_\cw$ are closed sectorial forms.
Let $A_W$ and $A_\cw$ be the operators associated with the forms
$\gota_W$ and $\gota_\cw$, respectively.
Roughly speaking, $A_W$ and $A_\cw$ are two versions of elliptic operators
with Dirichlet boundary conditions on~$D$ and Neumann boundary conditions
on $\partial \Omega \setminus D$.

\begin{thm} \label{tsector301}
For all $p \in [1,\infty]$ the semigroups generated by $-A_W$ and $-A_\cw$ extend consistently
to contraction semigroups on $L_p(\Omega)$, which are $C_0$-semigroups if $p \in [1,\infty)$
and they are bounded holomorphic if $p \in (1,\infty)$.
\end{thm}
\begin{proof}
It suffices to show that the semigroups generated by $-A_W$ and $-A_\cw$ are 
submarkovian.
Then the other statements follow by duality, interpolation,
and \cite{Ouh5} Proposition~3.12.

The semigroup generated by $-A_W$ is submarkovian by \cite{Ouh5} Corollary~4.10.
It follows from \cite{Ouh5} Corollary~4.10 and Theorem~2.13, applied to the 
operator with Neumann boundary conditions, that 
$(1 \wedge |u|) \sgn u \in W^{1,2}(\Omega)$ and 
\[
\RRe \gota [ (1 \wedge |u|) \sgn u ,(|u|-1)^+\sgn u ] \geq 0
\]
for all $u \in W^{1,2}(\Omega)$.
Hence one deduces that $(1 \wedge |u|) \sgn u \in \cw_D^{1,2}(\Omega)$
for all $u \in W^{1,2}(\Omega)$ with $D \cap \supp u = \emptyset$.
Since $ \{ u \in W^{1,2}(\Omega) : D \cap \supp u = \emptyset \} $ is 
dense in $\cw^{1,2}_D(\Omega)$, it follows from \cite{Ouh5} Theorem~2.13 3'$\Rightarrow$1
that the semigroup generated by $-A_\cw$ is submarkovian.
\end{proof}

For all $p \in [1,\infty)$ we denote by $-A_{W,p}$ and $-A_{\cw,p}$ the 
generator of the $C_0$-semigroup on $L_p(\Omega)$ 
which is consistent with the semigroup generated by $-A_W$ and $-A_\cw$.

We suppose that the reader is familiar with the concept of bounded $\ch^\infty$-calculus and refer
for details to \cite{Haase} and \cite{DHP}.
We wish to prove upper bounds for the $\ch^\infty$-angle of the operators
$A_{W,p}$ and $A_{\cw,p}$, first for $p = 2$ and then for all $p \in (1,\infty)$.

\begin{thm} \label{tsector302}
Adopt the assumptions and notation as in the beginning of this section.
\begin{tabel} 
\item \label{tsector302-1}
Suppose that $\one_\Omega \not\in W^{1,2}_D(\Omega)$.
Then for all $\varepsilon > 0$ the operator $A_W$ admits an $\ch^\infty$-calculus
on the sector $\Sigma(\kappa + \varepsilon)^\circ$.
Stronger, 
\[
\|f(A_W)\|_{2 \to 2} 
\leq \Big( 2 + \frac{2}{\sqrt{3}} \Big) \esssup_{z \in \Sigma(\kappa + \varepsilon)^\circ} |f(z)|
\]
for all $f \in \ch^\infty(\Sigma(\kappa + \varepsilon)^\circ)$.
\item \label{tsector302-2}
Suppose that $\one_\Omega \in W^{1,2}_D(\Omega)$.
Then for all $\delta,\varepsilon > 0$ the operator $A_W + \delta \, I$ admits an $\ch^\infty$-calculus
on the sector $\Sigma(\kappa + \varepsilon)^\circ$.
Stronger, 
\[
\|f(A_W + \delta \, I)\|_{2 \to 2} 
\leq \Big( 2 + \frac{2}{\sqrt{3}} \Big) \esssup_{z \in \Sigma(\kappa + \varepsilon)^\circ} |f(z)|
\]
for all $f \in \ch^\infty(\Sigma(\kappa + \varepsilon)^\circ)$.
\end{tabel}
Similar statements are valid for the operator $A_\cw$ instead of $A_W$.
\end{thm}
\begin{proof}
`\ref{tsector302-1}'.
Since $\Omega$ is connected, the operator $A_W$ is injective.
Moreover $\Lambda(A_W) \subset \Sigma(\kappa)$ by Corollary~\ref{csector203}.
Then the claim follows from the 
Crouzeix--Delyon theorem \cite{CrouzeixDelyon} Theorem~1.

`\ref{tsector302-2}'.
Let $\delta > 0$.
Then the operator $A_W + \delta \, I$ is injective
and $\Lambda(A_W + \delta \, I) \subset \Sigma(\kappa)$.
Then one can argue as in Statement~\ref{tsector302-1}.

The proof for $A_\cw$ is word-by-word the same.
\end{proof}

Next we consider the operators on $L_p(\Omega)$.

\begin{thm} \label{tsector303}
Adopt the assumptions and notation as in the beginning of this section.
Let $p \in (1,\infty)$ and set 
$\kappa_p = (1 - |1 - \frac{2}{p}|) \kappa + |1 - \frac{2}{p}| \, \frac{\pi}{2}$.
Then one has the following.
\begin{tabel} 
\item \label{tsector303-1}
Suppose that $\one_\Omega \not\in W^{1,2}_D(\Omega)$.
Then for all $\varepsilon > 0$ the operator $A_{W,p}$ admits an $\ch^\infty$-calculus
on the sector $\Sigma(\kappa_p + \varepsilon)^\circ$.
\item \label{tsector303-2}
Suppose that $\one_\Omega \in W^{1,2}_D(\Omega)$.
Then for all $\delta,\varepsilon > 0$ the operator $A_{W,p} + \delta \, I$ admits an $\ch^\infty$-calculus
on the sector $\Sigma(\kappa_p + \varepsilon)^\circ$.
\end{tabel}
Similar statements are valid for the operator $A_{\cw,p}$ instead of $A_{W,p}$.
\end{thm}
\begin{proof}
`\ref{tsector303-1}'.
The semigroup $(e^{-t A_W})_{t > 0}$ is positive by \cite{Ouh5} Corollary~4.3.
Hence by consistency and density the semigroup $(e^{-t A_{W,p}})_{t > 0}$ is positive
for all $p \in [1,\infty)$.
Since $- A_{W,p}$ is injective and the generator of a positive contraction semigroup,
it follows from Duong \cite{Duo2} Theorem~2 that for all $\theta \in (\frac{\pi}{2},\pi)$
the operator $A_{W,p}$ admits an $\ch^\infty$-calculus on the sector $\Sigma(\theta)^\circ$.

Finally we use interpolation. 
Suppose $p \in (2,\infty)$ and $\varepsilon > 0$.
For all $n \in \Ni$ with $n > p$ we interpolate between $2$ and $n$.
Let $\theta_n \in (0,1)$ be such that 
$\frac{1}{p} = \frac{\theta_n}{2} + \frac{1 - \theta_n}{n}$.
Then $\lim_{n \to \infty} \theta_n = \frac{2}{p}$. 
Now the operator $A_{W,2}$ admits an $\ch^\infty$-calculus on the sector $\Sigma(\kappa + \varepsilon)^\circ$
and the operator $A_{W,n}$ admits an $\ch^\infty$-calculus on the sector $\Sigma(\frac{\pi}{2} + \varepsilon)^\circ$.
One deduces from \cite{KKW} Proposition~4.9 that 
the operator $A_{W,p}$ admits an $\ch^\infty$-calculus on the sector 
$\Sigma(\theta_n \, \kappa + (1 - \theta_n) \, \frac{\pi}{2} + \varepsilon)^\circ$.
Taking $n$ large enough, the statement follows.

If $p \in (1,2)$, then the proof is similar, or one can use duality. 

`\ref{tsector303-2}'.
The proof is similar.

For the operators $A_{\cw,p}$ the argument is almost the same. 
The only thing that is not immediately clear is the positivity of the semigroup.
But that can be deduced as at the end of the proof of Theorem~\ref{tsector301}.
\end{proof}

We emphasise that $\kappa_p < \frac{\pi}{2}$ for all $p \in (1,\infty)$.
As an application of Theorem~\ref{tsector303} we obtain bounded imaginary powers 
with constant smaller than $\frac{\pi}{2}$.

\begin{cor} \label{csector304}
Adopt the assumptions and notation as in the beginning of this section.
Let $p \in (1,\infty)$ and let $\kappa_p \in [\kappa,\frac{\pi}{2})$ be as 
in Theorem~\ref{tsector303}.
Then for all $\varepsilon > 0$ there exists a $c > 0$ such that 
$\|(A_{W,p} + I)^{is}\|_{p \to p} \leq c \, e^{(\kappa_p + \varepsilon) |s|}$
and $\|(A_{\cw,p} + I)^{is}\|_{p \to p} \leq c \, e^{(\kappa_p + \varepsilon) |s|}$
for all $s \in \Ri$.
\end{cor}
\begin{proof}
Apply Theorem~\ref{tsector303} to the function $z \mapsto z^{is}$.
\end{proof}

\begin{cor} \label{csector305}
Adopt the assumptions and notation as in the beginning of this section.
For all $p \in (1,\infty)$ the operators $A_{W,p}$ and $A_{\cw,p}$ 
satisfy maximal parabolic regularity in~$L_p(\Omega)$.
\end{cor}
\begin{proof}
This follows from the Dore--Venni theorem \cite{DV} and Corollary~\ref{csector304}.
\end{proof}

We emphasise that $\Omega$ is merely an open connected (non-empty) set.
It does not need to have the doubling property, nor to be bounded.

\section{An outlook to more general measure spaces} \label{Ssector4}

In this section we consider a bounded domain and an elliptic 
operator with complex coefficients subject to 
mixed boundary conditions and domain inhomogeneities supported on a lower
dimensional hypersurface, enforcing a jump in the conormal derivative.

Let $\Omega \subset \Ri^d$ be open, bounded and connected.
Let $D \subset \partial \Omega$ be closed.
Further let $\Gamma_0$ be a Borel subset of $\Omega$ which is a 
$(d-1)$-set in the sense of Jonnson--Wallin (see \cite{JW} Subsection~VII.1.1), 
that is, there are $c_1,c_2 > 0$ such that 
\[
c_1 \, r^{d-1}
\leq \ch_{d-1}(B(x,r) \cap \Gamma_0) 
\leq c_2 \, r^{d-1}
\]
for all $x \in \Gamma_0$ and $r \in (0,1]$, where $\ch_{d-1}$ is the 
$(d-1)$-dimensional Hausdorff measure.
We emphasise that $\Gamma_0$ does not have to be closed.
Finally, we suppose that every element of $\overline{\partial \Omega \setminus D}$
admits a bi-Lipschitz chart. 

Set $\Gamma = \Gamma_0 \cup (\partial \Omega \setminus D)$.
Let $\rho$ be the restriction of $\ch_{d-1}$ to the set $\Gamma$.
Define $\Li_2 = L_2(\Omega \cup \Gamma, dx + d\rho)$.
There is a natural isomorphism from $\Li_2$ onto $L_2(\Omega,dx) \oplus L_2(\Gamma, d\rho)$.
We identify $\Li_2$ with $L_2(\Omega,dx) \oplus L_2(\Gamma, d\rho)$
in the natural way.
In this section we consider an m-sectorial operator in $\Li_2$.

For all $u \in L_{1,\loc}(\Omega)$ define the function $\Tr u$ by 
\[
\dom(\Tr u) 
= \{ x \in \Gamma : \lim_{r \downarrow 0} 
    \frac{1}{|\Omega \cap B(x,r)|} \int_{\Omega \cap B(x,r)} u(y) \, dy
      \mbox{ exists} \}
\]
and 
\[
(\Tr u)(x) = \lim_{r \downarrow 0} 
    \frac{1}{|\Omega \cap B(x,r)|} \int_{\Omega \cap B(x,r)} u(y) \, dy
\]
for all $x \in \dom(\Tr)$.
(See \cite{JW} Section~VIII.1.1.)
It follows from \cite{EMR} Proposition~2.8 that $\Tr u \in L_2(\Gamma,d\rho)$ 
for all $u \in W^{1,2}_D(\Omega)$.
Define $j \colon W^{1,2}_D(\Omega) \to \Li_2$ by 
\[
j(u) 
= (u, \Tr u)
\in L_2(\Omega,dx) \oplus L_2(\Gamma, d\rho)
 .  \]
By \cite{EMR} Lemma~2.10(i) the map $j$ is continuous and has dense range.


\begin{thm} \label{tsector401}
Adopt the above notation and assumptions.
Let $\mu \colon \Omega \to \Ci^{d \times d}$
be a bounded measurable function.
Suppose that there exists an $m > 0$ such that 
$\RRe \mu(x) \xi \cdot \overline \xi \geq m \, |\xi|^2$
for all $\xi \in \Ci^d$.
Define the sesquilinear form 
$\gota \colon W^{1,2}_D(\Omega) \times W^{1,2}_D(\Omega) \to \Ci$ by 
\[
\gota[u,v]= \int_\Omega \mu \nabla u \cdot \overline{\nabla v}
 .  \]
Define $\kappa = \arctan \sqrt{\bigl (\frac {M}{m}\bigr)^2- 1}$ and
$M= \esssup_{x \in \Omega} \|\mu(x)\|$.
Let $A$ be the m-sectorial operator in $\Li_2$ associated with $(\gota,j)$.
Then $\sigma(A) \subset \overline{\Lambda(A)} \subset \Sigma(\kappa)$.
Moreover, let $\theta \in (\kappa,\frac{\pi}{2})$.
Then 
\[
\|(A + \lambda \, I)^{-1}\|_{\Li_2 \to \Li_2}
\leq \frac{M}{ m \, \sin \theta - \sqrt{M^2 - m^2} \, \cos \theta} \, \frac{1}{|\lambda|}
\]
for all $\lambda \in \Sigma(\pi - \theta)$.
\end{thm}
\begin{proof}
This follows from Theorems~\ref{tsector202} and \ref{tsector210}.
\end{proof}

Let us mention that this provides an adequate functional analytic instrument 
to give equations with dynamical boundary conditions like 
\begin{alignat*}{3}
\partial_t u-\nabla \cdot \mu \nabla u & =  f_\Omega  & 
\qquad & \text{on }\,I \times (\Omega \setminus \Gamma_0) , \\
u & =  0  & &\text {on }\,  I \times D, \\
\partial_t u+\nu \cdot \mu \nabla u  & =  f_1 &  & 
\text{on }\, I \times (\partial \Omega \setminus D), \\
 \partial_t u+[\nu_{\Gamma_0} \cdot \mu \nabla u]  & =   f_0
  && \text{on }\, I \times \Gamma_0, \\
u(0) & =  u_0  & &\text{on }\, \Omega \cup (\partial \Omega \setminus D), 
\end{alignat*}
a precise meaning, inclusively its quasilinear variants, see \cite{EMR}.
For a strict derivation of dynamical boundary conditions in various
physical contexts see \cite{Gold1}.

\subsection*{Acknowledgements}
The first named author is grateful for hospitality
at the WIAS and he wishes to thank the WIAS for support.
Part of this work is supported by the Marsden Fund Council from Government
funding, administered by the Royal Society of New Zealand.

\bibliography{refbib}
\small 

\noindent
{\sc A.F.M. ter Elst,
Department of Mathematics,
University of Auckland,
Private bag 92019,
Auckland 1142,
New Zealand}  \\
{\em E-mail address}\/: {\bf terelst@math.auckland.ac.nz}

\mbox{}

\noindent
{\sc A. Linke,
Weierstrass Institute for Applied Analysis and Stochastics,
Mohrenstr.~39, 
10117 Berlin, 
Germany}  \\
{\em E-mail address}\/: {\bf alexander.linke@wias-berlin.de}

\mbox{}

\noindent
{\sc J. Rehberg,
Weierstrass Institute for Applied Analysis and Stochastics,
Mohrenstr.~39, 
10117 Berlin, 
Germany}  \\
{\em E-mail address}\/: {\bf rehberg@wias-berlin.de}

\end{document}